\documentclass[a4paper,12pt]{article}
\usepackage[top=3cm,bottom=3cm,left=2.5cm,right=2.5cm]{geometry}
\usepackage{t1enc}
\usepackage[latin1]{inputenc}
\usepackage[english]{babel}
\usepackage{amsmath,amsthm}
\usepackage{amsfonts}
\usepackage{latexsym}
\usepackage[dvips]{graphicx}
\usepackage{graphicx}
\usepackage{color}
\usepackage{amsmath}
\usepackage{amssymb}
\usepackage{amsbsy}
\usepackage{amsfonts}
\usepackage{amsfonts,mathrsfs}

\newtheorem{theorem}{Theorem}
\newtheorem{lemma}[theorem]{Lemma}

\begin{document}

\title{Maximum first Zagreb index of orientations of unicyclic graphs with given matching number}
\author{ Jiaxiang Yang, Hanyuan Deng\thanks{Corresponding author: hydeng@hunnu.edu.cn}\\
{\small Key Laboratory of Computing and Stochastic Mathematics (Ministry of Education),}
 \\{\small School of Mathematics and Statistics, Hunan Normal University,}
 \\{\small  Changsha, Hunan 410081, P. R. China.}}

\date{}
\maketitle

\begin{abstract}

 Let $D=(V,A)$  be a digraphs without isolated vertices. The first Zagreb index of a digraph $D$ is defined as a summation over all arcs,
$M_1(D)=\frac{1}{2}\sum\limits_{uv\in A}(d^{+}_{u}+d^{-}_v)$,
where $d^{+}_u$(resp. $d^{-}_u$) denotes the out-degree (resp. in-degree) of the vertex $u$.
In this paper, we give the maximal values and maximal digraphs of first Zagreb index over the set of all orientations of unicyclic graphs with $n$ vertices and matching number $m$ $(2\leq m\leq \lfloor \frac{n}{2}\rfloor)$.
\end{abstract}

\maketitle

\makeatletter
\renewcommand\@makefnmark%
{\mbox{\textsuperscript{\normalfont\@thefnmark)}}}
\makeatother

\baselineskip=0.25in

\section{Introduction}
 The first Zagreb index was first appeared in \cite {IG72,IG75}, and it is an important molecular descriptor which is related with many chemical properties. The first Zagreb index have been used in the study of molecular complexity, chirality, ZE-isomerism and heterosystems whilst the  Zagreb indices played a  potential role in applicability for deriving multilinear regression models. Zagreb indices are also used by  researchers in the studies of QSPR and QSAR \cite {SN03}. During the past decades, results closely correlated with  the Zagreb indices have published in \cite {AA18,BB17,KC03,KC04,KC10,KC09,BZ07}.

 We denote by $G=(V,E)$  a simple connected graph, where $V(G)$ is the vertex set of $G$ and $E(G)$ is the edge set of $G$. The first Zagreb index of $G$ is defined as
$$M_1(G)=\sum_{uv\in E(G)}(d_{G}(u)+d_{G}(v))$$
where $d_{G}(v)$ ($d_v$ for short) is the degree of vertex $v$ in $G$.

For any $v\in V(G)$, let $N_{G}(v) =\{u|uv\in E(G)\}$ be the neighbors of $v$,  and $d_{G}(v)\doteq |N_{G}(v)|$ the degree of $v$ in $G$.
 For $E'\subseteq E(G)$,  $G-E^{'}$ denotes the subgraph of $G$ obtained by deleting the edges of $E'$. Let $W\subseteq V(G)$, we denote by $G-W$  the subgraph of $G$ obtained by deleting the vertices of $W$ and the edges incident with them.
A matching $M$ of the graph $G$ is a subset of $E(G)$ such that no two edges in $M$ share a common vertex.  A matching $M$ of $G$ is  maximum, if  $|M_1| \leq |M|$ for any other matching $M_1$ of $G$. The matching number of $G$ is the number of edges
of a maximum matching in $G$. If $M$ is a matching of a graph $G$ and vertex $v\in V(G)$
is incident with an edge of $M$, then $v$ is said to be $M$-saturated, and if any $v\in V(G)$ is $M$-saturated, then $M$ is a perfect matching.

A digraph $D=(V,A)$ is an ordered pair $(V,A)$ consisting of a non-empty finite set $V$ of vertices and a
finite set $A$ of ordered pairs of distinct vertices called arcs (in particular, $D$ has no loops). Let $uv\in A$, we denote by $uv$  an
arc from vertex $u$ to vertex $v$. The vertex $u$ is the tail of $uv$, and
the vertex $v$ is its head.   $d^{+}_u$ (resp.$d^{-}_u$) denotes the out-degree (resp. in-degree) of a vertex $u$ which is the
number of arcs with tail $u$ (resp. with head $u$). If $u\in V$ and  $d^{+}_u=d^{-}_u=0$ ,then $u$ is called an isolated
vertex. $D_n$ denotes the set of all digraphs with $n$ non-isolated vertices.
The first Zagreb index of a digraph $D$ defined as
$$M_1(D)=\frac{1}{2}\sum_{uv\in A}(d^{+}_{u}+d^{-}_v)$$
where $d^{+}_u$(resp. $d^{-}_u$) denotes the out-degree(resp.in-degree) of the vertex $u$.
If $u\in V(D)$ and $d^{+}_u=0$ (resp. $d^{-}_u=0$), then $u$ is called a sink vertex (resp. source vertex).
An oriented graph $D$ is obtained from a graph $G$ by replacing each edge $uv$ of $G$ by an arc $uv$ or $vu$, but not both. In this case $D$ is also called an orientation of $G$. Let $\mathcal{O}(G)$ be the set of all orientations of $G$.
$D\in \mathcal{O}(G)$,if $d^{+}_u=0$ or $d^{-}_u=0$ for any $u\in V(D)$, then $D$ is called a
sink-source orientation of $G$.\\
\indent In order to better study of vertex-degree-based topological indices.   Recently, J. Monsalve and J. Rada \cite {JM211} extended the concept of vertex-degree based topological indices of graphs to oriented graphs.
the authors determined the extremal values of the Randi$\acute{c}$ index over $\mathcal{OT}(n)$, the set of all oriented trees with $n$ vertices. Also,
the authors given the extremal values of the Randi$\acute{c}$ index over $\mathcal{O}(P_n)$,$\mathcal{O}(C_n)$ and $\mathcal{O}(H_d)$, where $P_n$ is the path with $n$, $P_n$ is the cycle
with $n$ vertices and $H_d$ is the hypercube of dimension $d$, respectively.
J. Monsalve and J. Rada \cite {JM212}
found extremal values of symmetric VDB topological
indices over $\mathcal{OT}(n)$ and $\mathcal{O}(G)$, respectively. But the maximum value of $\mathcal{AZ}$ over $\mathcal{OT}(n)$ is still an open problem.

In this paper, we present the maximal first Zagreb  index for orientations of unicyclic graphs with $n$ vertices and
matching number $m$ ($2\leq m\leq \lfloor\frac{n}{2}\rfloor$), and
we state the results as follows:\\
\indent Let $n$ and $m$ be integers and  $2\leq m\leq \lfloor\frac{n}{2}\rfloor$,  $U(n,m)$  the class of unicyclic graphs on $n$ vertices with matching number $m$, and  $U_{n,m}$  the graph formed by attaching $n-2m+1$ pendent vertices and $m-2$ paths of length 2 to a (common) vertex of a triangle. Let $U^{(1)}_{n,m},U^{(2)}_{n,m},U^{(3)}_{n,m},U^{(4)}_{n,m}$ be  four orientations of $U_{n,m}$ (see Figure $\ref{fig-3}$). Obviously, $U_{n,m}\in U(n,m)$. Let $C_n$ be the the cycle with $n$ vertices.  $U^{*}_{4,2}=\{U^{(1)}_{4,2},U^{(2)}_{4,2},U^{(3)}_{4,2},U^{(4)}_{4,2},U^{(5)}_{4,2},U^{(6)}_{4,2}\}$, where $U^{(5)}_{4,2}$ and $U^{(6)}_{4,2}$ are the sink-sourse orientations of $C_4$. $U^{*}_{6,3}=\{U^{(1)}_{6,3}, U^{(2)}_{6,3}, U^{(3)}_{6,3}, U^{(4)}_{6,3}, U^{(5)}_{6,3},\\
 U^{(6)}_{6,3}\}$, where $U^{(5)}_{6,3}$ and $U^{(6)}_{6,3}$ are the sink-sourse orientations of the graph formed by attaching two pendant vertices to two adjacent vertices of $C_4$. $U^{*}_{n,m}=\{U^{(1)}_{n,m},U^{(2)}_{n,m},U^{(3)}_{n,m},U^{(4)}_{n,m}\}$, where $(n,m) \neq(4,2),(6,3)$.

\begin{figure}[ht]
\begin{center}
  \includegraphics[width=15cm,height=3cm]{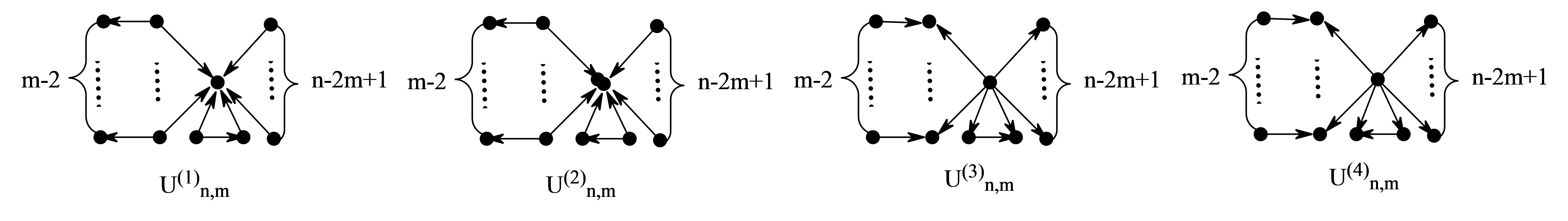}
     \end{center}
\vskip -0.5cm
\caption{ Four orientations of $U_{n,m}$ :$U^{(1)}_{n,m}$,$U^{(2)}_{n,m}$,$U^{(3)}_{n,m}$,$U^{(4)}_{n,m}$ .}\label{fig-3}
\end{figure}

\begin{theorem}\label{the11}
 Let $G\in U(n,m)$ with $2\leq m\leq \lfloor\frac{n}{2}\rfloor$, $D\in \mathcal{O}(G)$. Then
$$M_{1}(D)\leq \frac{1}{2}\left[n^{2}+(-2m+3) n+m^{2}+m-2\right]$$
with equality if and only if
$D\in U^{*}_{n,m}$.
\end{theorem}

\indent Specially, if $n=2m$, we have

\begin{theorem}\label{the8}
  Let $G \in U(2m, m)$ with $m\geq 2$, $D\in \mathcal{O}(G)$.
 Then $$M_1(D)\leq \frac{1}{2}[m^2+7m-2]$$
with equality if and only if $G\in U^{*}_{2m,m} $.
\end{theorem}
\indent Hence, we solve the problem on the maximum values of the first Zagreb  index for orientations of unicyclic graphs with $n$ vertices and
matching number $m$ ($2\leq m\leq \lfloor\frac{n}{2}\rfloor$).

\section{Some useful lemmas }

 In this section, we give three useful lemmas.

\begin{lemma}\cite {JM19}\label{lem1}
 Let $G$ be a graph. Then $G$ is a bipartite graph if and only if $G$ has a sink-source orientation. Moreover, If $G$ is a connected bipartite graph, then there exist exactly two sink- source orientations of $G$.
\end{lemma}

Now, we can show a important result.

\begin{lemma}\label{lem2}
Let $G$ be a graph, $D \in \mathcal{O}(G)$. Then
$$M_1(D)\leq \frac{M_1(G)}{2}$$ equality occurs if and only if $D$ is a sink-source orientation of $G$.

\end{lemma}
\begin{proof}
Let $G = (V,E)$ and $D=(V,A)$. For each $u \in V$,
$d_u=d_u^{+}+d_u^{-}$.
So $d_u\geq d_u^{+}$ and $d_v \geq d_v^{-}$, where $u,v \in V$. Then $d_u^{+}+d_v^{-}\leq  d_{u}+d_v$. Hence
$$
M_1(D)=\frac{1}{2}\sum_{uv\in A}(d_u^{+}+d_v^{-})\leq \frac{1}{2}\sum_{uv\in E(G)}(d_u+d_v)=\frac{M_1(G)}{2}.
$$
\indent If $D$ is a sink-source orientation of $G$, then for each $u \in V$, one has either $d_u^{+}=0$ or $d_u^{-}=0$. Moreover, if $uv \in A$, then
$d_u^{+}\neq 0$ and  $d_u^{-}=0$, so $d_u=d_u^{+}$. It is Similar to $d_v$. Hence
$$
M_1(D)=\frac{1}{2}\sum_{uv\in A}(d_u^{+}+d_v^{-})= \frac{1}{2}\sum_{uv\in E(G)}(d_u+d_v)=\frac{M_1(G)}{2}.
$$
\indent Conversely,  $d_u \geq d_u^{+}$ and $d_v \geq d_v^{-}$, then $d_u+d_v \geq d_u^{+}+d_v^{-}$ with equality  if and only if $d_u =d_u^{+}$ and $d_v=d_v^{-}$, so $M_1(D)=\frac{M_1(G)}{2}$ if and only if $d_u=d_u^{+}$
and $d_v=d_v^{-}
$, where  all $uv \in A$. This clearly implies that either  $d_w^{+}=0$ or $d_w^{-}=0$ for any $w\in V$.\\
\end{proof}
\begin{lemma}\label{lem6}
  Let $G$ be the graph with $n$ non-islated vertices and $D\in \mathcal{O}(G)$. Then $$M_{1}(D)=\frac{1}{2} \sum_{u\in V(D)}\left[(d^{+}_u)^{2}+(d^{-}_u)^{2}\right]$$

\end{lemma}

\begin{proof}
 As the fact that $M_{1}(D)=\frac{1}{2} \sum \limits_{uv\in A}\left[(d^{+}_u)+(d^{-}_v)\right]$ and $d^{+}_u$  (resp. $d^{-}_u$) occur $d^{+}_u$ (resp. $d^{-}_u$) times in the sum, for each $u\in V(D)$. \\
 \indent So, $M_{1}(D)=\frac{1}{2} \sum_{u\in V(D)}\left[(d^{+}_u)^{2}+(d^{-}_u)^{2}\right]$.
\end{proof}

\section{Proof of Theorem 2}

 In this section, we first give a proof of Theorem 2, then we will prove Theorem 1 in next section by using Theorem 2.\\
 \indent We first determine the maximum  values of the first Zagreb  index for orientations of trees with $2m$ vertices and
matching number $m$ ($m\geq 1$). \\
\indent Let $n$ and $m$ be integers and $1\leq m\leq \lfloor\frac{n}{2} \rfloor$. $T(n,m)$ denotes  the class of trees on $n$ vertices with matching number $m$. We denote by $T_{n,m}$  a tree formed by attaching a pendent vertex to each of $m-1$ pendent vertices of the graph $K_{1,n-m}$, where a pendent vertex is a vertex of degree one (see Figure $\ref{fig-1}$).  Obviously, $T_{n,1} =K_{1,n-1}$ and $T_{n,m}\in T(n,m)$. Let $T$ be a tree with $u,v \in V(T)$. We  denote by $P_{T}(u,v)$  the unique path from $u$ to $v$ in $T$. Firstly, we give a  lemma which is related to $P_{T}(u,v)$.
\begin{figure}[ht]
\begin{center}
  \includegraphics[width=5cm,height=2cm]{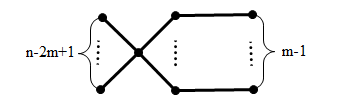}
     \end{center}
\vskip -0.5cm
\caption{ The graph $T_{n,m}$ .}\label{fig-1}
\end{figure}
\begin{lemma}\cite{WL10}\label{lem31}
Let $T$ be a tree with at least four vertices and a perfect matching $M$. If $P_{T}(u,v)$ as a diametrical path in $T$, then the unique neighbor of $u$ has degree two.
\end{lemma}

We first consider trees with a perfect matching.

\begin{lemma}\label{lem3}
Let $T\in T(2m,m)$ with $m\geq 1$. Then
$$M_1(T)\leq m^2+5m-4$$
with equality if and only $T \cong T_{2m,m}$.
\end{lemma}
\begin{proof}
We will prove by induction on $m$. \\
\indent Obviously, $T=T_{2,1}$ for $m=1$, and $T = T_{4,2}$ for $m =2$. So the result holds for $m=1,2$.\\
\indent If $m\geq 3$.  Suppose that the result holds for trees in $T(2(m-1),(m-1))$. Let $T\in T(2m,m)$ and $M$  a perfect matching of $T$. Note that the diameter of $T$ is at least four. We can denote by  $P_{T}(u,v)=uxyz\cdot\cdot\cdot$  a diametrical path in $T$. Then $z\neq v$. Let $N_{T}(y)=\{x_1,x_2,\cdot\cdot\cdot,x_{s+1}\}$ with $x_1 = x$ and $x_{s+1}=z$.\\
\indent Suppose that $yz\in M$. By Lemma 6,  $d_{x_i}= 2$ and $d_{u_i}= 1$, where $u_i$ is the neighbor of $x_i$ different from $y$ for $1\leq i\leq s$, and $u_1=u$. So $2(2m-1)\geq \sum\limits_{i=1}^{s}(d_{x_i}+d_{u_i})+d_y+d_z+d_v\geq 3s+(s+1)+2+1>4s+2$, hence $s<m-1$. Suppose that $yz\notin M$. Then $M$ contains $zw$ for some neighbor $w$ of $z$ different from $y$, and $M$ contains one of $yx_i$ for $2 \leq i\leq s$, say $yx_s$. Since $P_{T}(u,v)$ is a diametrical path, $x_s$ is a pendent vertex. By Lemma \ref{lem31}, $d_{x_j}=2$ and $d_{u_j}= 1$, where $u_j$ is the neighbor of $x_j$ different from $y$ for $1\leq j \leq s-1$, and $u_1=u$. So $2(2m-1)\geq \sum\limits_{j=1}^{s-1}(d_{x_j}+d_{u_j})+d_{x_s}+d_z+d_y+d_w\geq 3(s-1)+1+(s+1)+2+1=4s+2$, hence $s \leq m-1.$ Consequently, $s \leq m-1$.\\
\indent Let $T'=T-\{u,x\}\in T(2(m-1),m-1)$ and  it is easily checked that $M-\{ux\}$ is  a perfect matching of $T'$. \\
\indent By the induction hypothesis, it is obvious that $M_1(T')\leq (m-1)^2+5(m-1)-4$. Hence
\begin{equation*}
\begin{aligned}
M_1(T) \leqslant& M_1(T')+d_x+d_u+d_y+d_{x}+\sum_{i=2}^{s}\left[\left(d_{y}+d_{ x_i}\right)-\left(d_{y}-1+d_{x_i}\right)\right]\\
&+[(d_y+d_z)-(d_{y}-1+d_z)]
\end{aligned}
\end{equation*}
\begin{equation*}
\begin{aligned}
& \leqslant M_1(T')+3+(s+3)+s \\
& \leqslant (m-1)^2+5(m-1)-4+6+2(m-1)\\
&=m^2+5m-4, 
\end{aligned}
\end{equation*}
equality  occurs if and only if $M_1(T')=(m-1)^2+5(m-1)-4$ and $s=m-1$  or equivalently, $T-\{u,x\}\cong T_{2(m-1),m-1}$, $yz\notin M$ and $d_y=m$, i.e. $T\cong T_{2m,m}$.
\end{proof}

We can extend the result for the first Zagreb index of trees to the oriented trees.
\begin{lemma}\label{lem4}
Let $T\in T(2m,m)$ with $m\geq 1$, $D\in \mathcal{O}(T)$. Then
$$M_1(D)\leq \frac{1}{2}(m^2+5m-4)$$
with equality if and only $D$ is a sink-source orientation of $T_{2m,m}$.
\end{lemma}
\begin{proof}
Let $D\in \mathcal{O}(T)$, where $T\in T(2m,m)$.
    Since $T$ is a bipartite graph, $T$ has sink-source orientation, by Lemma $\ref{lem1}$.\\
     \indent From Lemma $\ref{lem2}$, $M_{1}(D)\leq \frac{1}{2}M_{1}(T)$, equality occurs if ond only if $D$ is a sink-source orientation of $T$.\\
     \indent Hence, by Lemma $\ref{lem3}$, $$\max \{M_{1}(D)|\\D\in \mathcal{O}(T),T\in T(2m,m)\}=\max \{\frac{1}{2}M_{1}(T)|T\in T(2m,m)\}=\frac{1}{2}M_{1}(T_{2m,m})$$
    \indent Consequently,  $M_1(D)\leq \frac{1}{2}(m^2+5m-4)$,
    equality occurs if and only if $D$ is a sink-source orientation of $T_{2m,m}$.

\end{proof}

We give the maximum  values of the first Zagreb  index for orientations of two graph, which will be used in the following.

\begin{lemma}\label{lem5}
Let $D\in \mathcal{O}(U_{4,2})$. Then $$M_{1}(D)\leq 8$$
with equality if and only if  $D\in\{U^{(1)}_{4,2},U^{(2)}_{4,2},U^{(3)}_{4,2},U^{(4)}_{4,2}\}$.
\end{lemma}

\begin{figure}[ht]
\begin{center}
  \includegraphics[width=13cm,height=4cm]{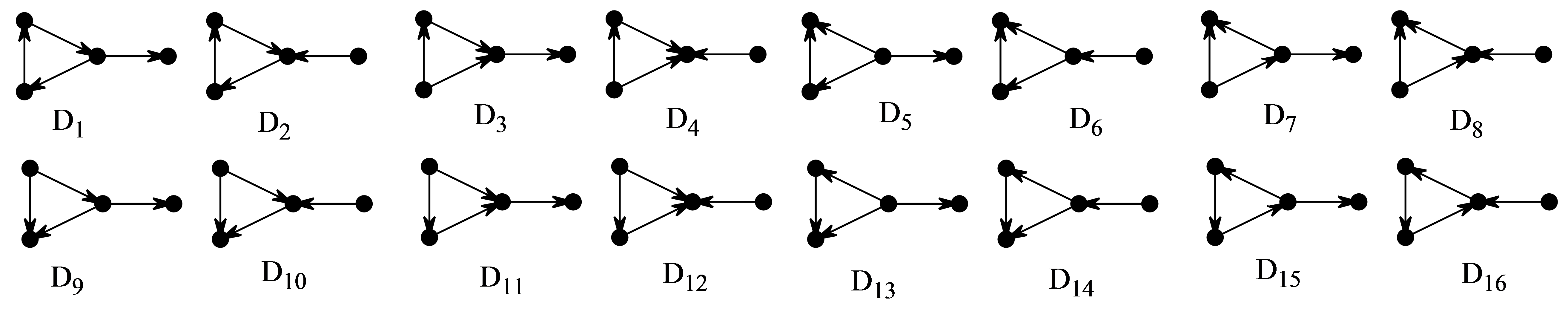}
     \end{center}
\vskip -0.5cm
\caption{ $D_i$,$i=1,2,\cdot\cdot\cdot,16$.}\label{fig-4}
\end{figure}

\begin{proof}
 Let $D\in \mathcal{O}(U_{4,2})$. Since each $uv\in E(U_{4,2})$, $uv$ has two orientations and $|E(U_{4,2})|=4 $, we
  have $|\mathcal{O}(U_{4,2})|=2^4=16$. Note that $\mathcal{O}(U_{4,2})=\{D_1,D_2,\cdot\cdot\cdot,D_{16}\}$  (see Figure $\ref{fig-4}$).\\
     \indent Clearly, we have
     $$M_{1}(D_1)=M_{1}(D_2)=M_{1}(D_{15})=M_{1}(D_{16})=5$$
     $$M_{1}(D_3)=M_{1}(D_6)=M_{1}(D_{11})=M_{1}(D_{14})=6$$
     $$M_{1}(D_7)=M_{1}(D_8)=M_{1}(D_9)=M_{1}(D_{10})=7$$
     $$M_{1}(D_4)=M_{1}(D_5)=M_{1}(D_{12})=M_{1}(D_{13})=8$$\\
     \indent Consequently, $M_{1}(D)\leq 8$, equality occurs if and only if $D\in \{D_4,D_5,D_{12},D_{13}\}\\=\{U^{(1)}_{4,2},U^{(2)}_{4,2},U^{(3)}_{4,2},U^{(4)}_{4,2}\}$.

\end{proof}

\begin{lemma}\label{lem7}
  Let $G_1$ be the graph formed by attaching a pendent vertex to each vertex of a triangle. Let $D\in \mathcal{O}(G_1)$. Then $$M_1(D)\leq 13$$
with equality if and only if  $D\in \{D_{12},D_{21},D_{23},D_{24},D_{28},D_{32},D_{33},D_{37},D_{41},D_{42},D_{44},D_{53}\}$
(see Figure $\ref{fig-5}$).
\end{lemma}

\begin{figure}[ht]
\begin{center}
  \includegraphics[width=15cm,height=10cm]{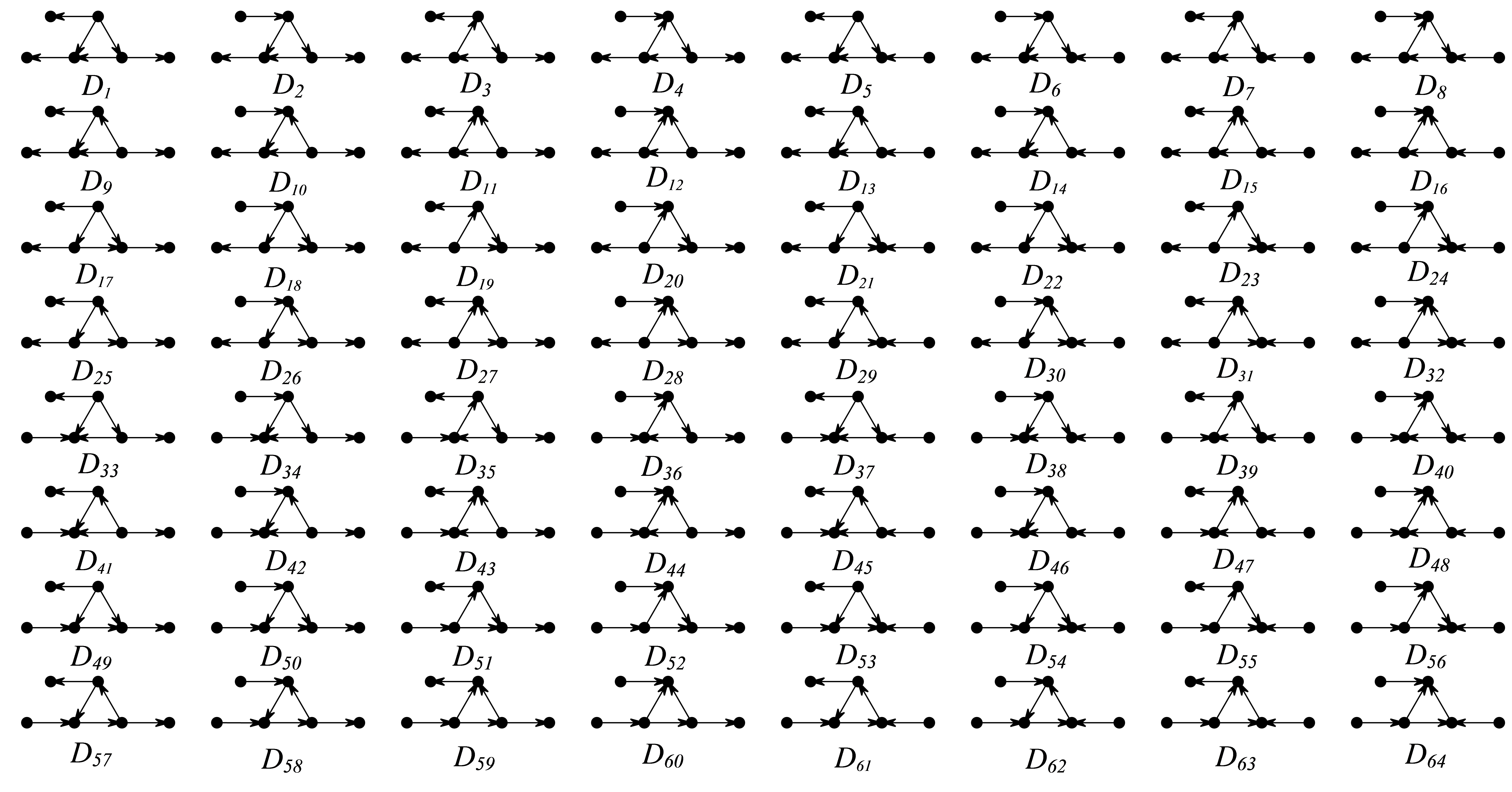}
     \end{center}
\vskip -0.5cm
\caption{ $D_i$, $i=1,2,\cdot\cdot\cdot,64$ in Lemma \ref{lem7} .}\label{fig-5}
\end{figure}

\begin{proof}
Note that $\mathcal{O}(G_1)=\{D_{i}|i=1,2,\cdot\cdot\cdot,64\}$ (see Figure $\ref{fig-5}$). \\
\indent Let $u_i$ be a pendent vertex and $v_i$  the unique neighbor of $u_i$ in $G_1$, where $i=1,2,3$. \\
\indent Obviously, all digraphs in Figure $\ref{fig-5}$ have $\{d^{+}_{u_{i}}=1,d^{-}_{u_{i}}=0\}$ or $\{d^{+}_{u_{i}}=0,d^{-}_{u_{i}}=1\}$, where $i=1,2,3$. By Lemma $\ref{lem6}$,
 \begin{align*}
 M_{1}(D_j)&=\frac{1}{2} \sum_{i=1}^{3}\left[(d^{+}_{D_j} (u_i))^2+(d^{-}_{D_j} (u_i))^2\right]+\frac{1}{2} \sum_{i=1}^{3}\left[(d^{+}_{D_j} (v_i))^2+(d^{-}_{D_j} v_i)^2\right]\\
 &=\frac{1}{2} [3+\sum_{i=1}^{3}\left[(d^{+}_{D_j} (v_i))^2+(d^{-} _{D_j}(v_i))^2\right]], 
 \end{align*}
where $j=1,2,\cdot\cdot\cdot,64$.
All digraphs in Figure $\ref{fig-5}$ can be divided  into three case:
\\ \indent Case 1.$\{d^{+}_{v_{i}}=2,d^{-}_{v_{i}}=1\}$ or $\{d^{+}_{v_{i}}=1,d^{-}_{v_{i}}=2\}$, where $i=1,2,3$.
This clearly implies that $M_{1}(D_2)=M_{1}(D_3)=M_{1}(D_4)=M_{1}(D_8)=M_{1}(D_7)=M_{1}(D_6)\\
=M_{1}(D_{13})=M_{1}(D_{14})=M_{1}(D_{15})=M_{1}(D_{18})=M_{1}(D_{25})=M_{1}(D_{26})\\=
M_{1}(D_{29})=M_{1}(D_{30})=M_{1}(D_{35})=M_{1}(D_{36})=M_{1}(D_{39})=M_{1}(D_{40})\\=
M_{1}(D_{47})=M_{1}(D_{50})=M_{1}(D_{51})=M_{1}(D_{52})=M_{1}(D_{57})=M_{1}(D_{58})\\=
M_{1}(D_{59})=M_{1}(D_{61})=M_{1}(D_{62})=M_{1}(D_{63})=9$
\\ \indent Case 2. There is a $v_i$ which satisfy $\{d^{+}_{v_{i}}=3,d^{-}_{v_{i}}=0\}$ or $\{d^{+}_{v_{i}}=0,d^{-}_{v_{i}}=3\}$, says $v_1$. $\{d^{+}_{v_{i}}=2,d^{-}_{v_{i}}=1\}$ or $\{d^{+}_{v_{i}}=1,d^{-}_{v_{i}}=2\}$, where $i=2,3$.
This clearly implies that $M_{1}(D_1)=M_{1}(D_5)=M_{1}(D_9)=M_{1}(D_{10})=M_{1}(D_{11})=M_{1}(D_{16})\\
=M_{1}(D_{17})=M_{1}(D_{19})=M_{1}(D_{20})=M_{1}(D_{22})=M_{1}(D_{27})=M_{1}(D_{31})\\=
M_{1}(D_{34})=M_{1}(D_{38})=M_{1}(D_{43})=M_{1}(D_{45})=M_{1}(D_{46})=M_{1}(D_{48})\\=
M_{1}(D_{49})=M_{1}(D_{54})=M_{1}(D_{55})=M_{1}(D_{56})=M_{1}(D_{60})=M_{1}(D_{64})=11$
\\ \indent Case 3. There are two $v_i$ satisfy $\{d^{+}_{v_{i}}=3, d^{-}_{v_{i}}=0\}$ or $\{d^{+}_{v_{i}}=0, d^{-}_{v_{i}}=3\}$, say $v_1$ and $v_2$. $\{d^{+}_{v_{3}}=1,d^{-}_{v_{3}}=2\}$ or $\{d^{+}_{v_{3}}=2,d^{-}_{v_{3}}=1\}$.
This clearly implies that
$M_{1}(D_{12})=M_{1}(D_{21})=M_{1}(D_{23})=M_{1}(D_{24})=M_{1}(D_{28})=M_{1}(D_{32})\\
=M_{1}(D_{33})=M_{1}(D_{37})=M_{1}(D_{41})=M_{1}(D_{42})=M_{1}(D_{44})=M_{1}(D_{53})=13$\\
\indent Consequantly, $M_1(D)\leq 13$, equality occurs if and only if $D\in \{D_{12},D_{21},D_{23},D_{24},\\D_{28},D_{32},D_{33},D_{37},D_{41},D_{42},D_{44},D_{53}\}.$
\end{proof}

We are now ready to give a proof of Theorem 2.\\
\textbf{Proof of Theorem 2}.
\begin{proof}
 We will prove by induction on $m$.\\
\indent If $m=2$, then either $G=C_4 $ or $G=U_{4,2}$, and by Lemma $\ref{lem2}$, $D\in \mathcal{O}(C_4)$, $M_1(D)\leq \frac{1}{2}M_1(C_4)=8=\frac{1}{2}(2^2+7\times 2-2)$, equality occurs if and only if $D$ is a sink-source orientation of $C_4$, i.e., $U^{(5)}_{4,2}$ or $U^{(6)}_{4,2}$. By Lemma $\ref{lem5}$, $D\in \mathcal{O}(U_{4,2})$, $M_1(D)\leq 8=\frac{1}{2}[2^2+7\times 2-2]$, equality occurs if and only if $D\in \{U^{(1)}_{4,2},U^{(2)}_{4,2},U^{(3)}_{4,2},U^{(4)}_{4,2}\}$. Consequently, $G\in U(4,2)$, $D\in \mathcal{O}(G)$, $M_1(D)\leq 8$, equality occurs if and only if $D\in \{U^{(1)}_{4,2},U^{(2)}_{4,2},U^{(3)}_{4,2},U^{(4)}_{4,2},U^{(5)}_{4,2}, U^{(6)}_{4,2}\}=U^{*}_{4,2}$.
The result holds.

 If $m \geq 3$. Suppose that the result holds for all orientations of unicyclic graphs in $ U(2(m - 1),m - 1)$. \\
\indent Let $G \in U(2m,m)$ with a perfect matching $M$.
If $G = C_{2m}$, then $D\in \mathcal{O}(C_{2m})$, by Lemma $\ref{lem2}$ and $\frac{1}{2}(m^2+7m-2)-4m=\frac{1}{2}(m^2-m-2)=\frac{1}{2}[(m-\frac{1}{2})^2-\frac{9}{4}]\geq 2>0$, $M_1(D)\leq \frac{1}{2}M_1(C_{2m})=4m <\frac{1}{2}[m^2+7m-2]$. The result holds.

Suppose that $G\neq C_{2m}$, we consider the following two cases.\\
\indent Case 1. Suppose that $G$ has a pendent vertex $u$ whose unique neighbor $v$ has degree two. Let $w\in N_{G}(v)$ and $w\neq u$.  Obviously, $d_{w}\geq 2$. Let $N_{G}(w) = \{v_1,v_2,\cdot\cdot\cdot,v_{s+1}\}$,
where $s \geq 1$ and $v_1 = v$. Then $M$ contains one of $wv_i$, $i = 2,3,\cdot\cdot\cdot,s+1$, say $wv_2$. Since the $s-1$ vertices $v_3,\cdot\cdot\cdot,v_{s+1}$ are $M$-saturated and at most two of them belong to the unicyclic component of $G-\{w\}$, we have $m\geq2+(s-2)=s$.
Then $G'=G-\{u,v\} \in U(2(m-1),m-1)$ and $M-\{uv\}$ is a perfect matching of $G'$. Let $D'\in \mathcal{O}(G')$ and $A(D')\bigcap A(D)=A(D')$, where $D\in \mathcal{O}(G)$.

By the induction hypothesis, it is obvious that  $M_1(D')\leq \frac{1}{2}[(m-1)^2+7(m-1)-2]$.\\
\indent
If $uv\in A(D)$, then $\frac{1}{2}[d_{D}^{+}(u)+d_{D}^{-}(v)]\leq \frac{1}{2}[d_{G}(u)+d_{G}(v)]$.
If $vu\in A(D)$, then $\frac{1}{2}[d_{D}^{-}(u)+d_{D}^{+}(v)]\leq \frac{1}{2}[d_{G}(u)+d_{G}(v)]$.
Hence, $\max \{\frac{1}{2}[d_{D}^{+}(u)+d_{D}^{-}(v)],\frac{1}{2}[d_{D}^{-}(u)+d_{D}^{+}(v)]\}\leq \frac{1}{2}[d_{G}(u)+d_{G}(v)]$ . Similarly to $vw\in A$ and $wv\in A$, and we have $\max \{\frac{1}{2}[d_{D}^{+}(v)+d_{D}^{-}(w)],\frac{1}{2}[d_{D}^{-}(v)+d_{D}^{+}(w)]\}\leq \frac{1}{2}[d_{G}(v)+d_{G}(w)]$.\\
  \indent If $vw\in A(D)$, then $d_{D}^{-}(w)=d_{D'}^{-}(w)+1$,$d_{D}^{+}(w)=d_{D'}^{+}(w)$. Since $A(D')\cap A(D)=A(D')$,   without lost of generality suppose that $d_{D}^{+}(v_i)=d_{D'}^{+}(v_i)$, where $i=2,3,\cdot\cdot\cdot,d_{D}^{-}(w)$.  $d_{D}^{-}(v_j)=d_{D'}^{-}(v_j)$, where  $j=d_{D}^{-}(w)+1,\cdot\cdot\cdot,d_{G}(w)$. Consequently    $d_{D}^{+}(v_i)+d_{D}^{-}(w)=d_{D'}^{+}(v_i)+d_{D'}^{-}(w)+1$,
  where  $i=2,3,\cdot\cdot\cdot,d_{D}^{-}(w)$.   $d_{D}^{-}(v_j)+d_{D}^{+}(w)=d_{D'}^{-}(v_j)+d_{D'}^{+}(w)$, where   $j=d_{D}^{-}(w)+1,\cdot\cdot\cdot,d_{G}(w)$. Similarly to $wv\in A(D)$. Thus



\begin{equation*}
\begin{aligned}
M_{1}(D)  \leqslant& M_{1}(D')+\max \{\frac{1}{2}[d_{D}^{+}(u)+d_{D}^{-}(v)],\frac{1}{2}[d_{D}^{-}(u)+d_{D}^{+}(v)]\}+\max \{\frac{1}{2}[d_{D}^{+}(w)+d_{D}^{-}(v)],\\
&\frac{1}{2}[d_{D}^{-}(w)+d_{D}^{+}(v)]\}
+\frac{1}{2}\max\{\sum_{i=2}^{d_{D}^{+}(w)}[d_{D}^{-}(v_i)+d_{D}^{+}(w)-(d_{D'}^{-}(v_i)+d_{D'}^{+}(w))]
\\
&+\sum_{j=d_{D}^{+}(w)+1}^{d_{G}(w)}[d_{D}^{+}(v_j)+d_{D}^{-}(w)-(d_{D'}^{+}(v_j)+d_{D'}^{-}(w))],
\sum_{i=2}^{d_{D}^{-}(w)}[d_{D}^{+}(v_i)+d_{D}^{-}(w)\\
&-(d_{D'}^{+}(v_i)+d_{D'}^{-}(w))]+\sum_{j=d_{D}^{-}(w)+1}^{d_{G}(w)}[d_{D}^{-}(v_j)+d_{D}^{+}(w)-(d_{D'}^{-}(v_j)+d_{D'}^{+}(w))]\}
\end{aligned}
\end{equation*}

\begin{equation*}
\begin{aligned}
& \leqslant M_{1}(D')+\frac{1}{2}[d_{G}(u)+d_{G}(v)]+\frac{1}{2}[d_{G}(w)+d_{G}(v)]+\frac{1}{2}\max\{d_{D}^{+}(w)-1,d_{D}^{-}(w)-1\}\\
& \leqslant M_{1}(D')+\frac{1}{2}[d_{G}(u)+d_{G}(v)]+\frac{1}{2}[d_{G}(w)+d_{G}(v)]+\frac{1}{2}(d_{G}(w)-1) \\
& \leqslant M_{1}(D')+s+3 \\
& \leqslant \frac{1}{2}\left[(m-1)^{2}+7(m-1)-2\right]+m+3 \\
& = \frac{1}{2}\left[m^{2}+7m-2\right], 
\end{aligned}
\end{equation*}
equality occurs if and only if $M_{1}(D')= \frac{1}{2}\left[(m-1)^{2}+7(m-1)-2\right]$ ,\\ $\max \{\frac{1}{2}[d_{D}^{+}(u)+d_{D}^{-}(v)],\frac{1}{2}[d_{D}^{-}(u)+d_{D}^{+}(v)]\}=\frac{1}{2}[d_{G}(u)+d_{G}(v)]$,  $\max \{\frac{1}{2}[d_{D}^{+}(w)+d_{D}^{-}(v)],\frac{1}{2}[d_{D}^{-}(w)+d_{D}^{+}(v)]\}=\frac{1}{2}[d_{G}(w)+d_{G}(v)]$, $\frac{1}{2}\max \{[d_{D}^{+}(w)-1],[d_{D}^{-}(w)-1]\}=\frac{1}{2} [d_{G}(w)-1]$ and $s=m$,
or equivalently, $D'\in U_{2(m-1),(m-1)}^{*}$ and $\{d_{D}^{+}(w)=m+1,d^{-}_{D}(v)=d_{G}(v),d^{+}_{D}(u)=d_{G}(u)\}$ or $\{d_{D}^{-}(w)=m+1,d^{+}_{D}(v)=d_{G}(v),d^{-}_{D}(u)=d_{G}(u)\}$, i.e. $D\in U_{2m,m}^{*}$. The result holds.

Case 2. Suppose that $G$ has a pendent vertex $u$ and $d_{v}\neq 2$ for $v\in N_{G}(u)$.  $C = v_{1}v_2...v_{t}v_{1}$ denotes the unique cycle of $G$. Since $M$  is a perfect matching of $G$, $G-V(C)$ consists of isolated vertices.

Subcase 2.1. If each vertex of $C$ is adjacent to a pendent vertex in $G$. Then $D\in \mathcal{O}(G)$. When $m\geq 4$, by Lemma $\ref{lem2}$ and $\frac{1}{2}[m^2+7m-2]-5m=\frac{1}{2}[m^2-3m-2]=\frac{1}{2}[(m-\frac{3}{2})^2-\frac{17}{8}]>0$, we have $M_1(D)\leq \frac{1}{2}M_1(G)=5m<\frac{1}{2}[m^2+7m-2]$.
When $m=3$, by Lemma $\ref{lem7}$, $M_1(D)\leq 13<14=\frac{1}{2}(3^2+7\times 3-2)$. The result holds.\\
\indent Subcase 2.2. Suppose that there is at least one vertex of degree two on $C$. Obviously, $d_{v_1}=2$ or 3.  Without lost of generality suppose that $d_{v_2}=3$ and $d_{v_3}=2$.
Let $u_{2}\in N_{G}(v_2)$ and $d_{u_2}=1$. Since $v_{2}u_{2} \in M$ and $v_3$ is $M$-saturated, we have $v_{3}v_{4}\in M$ and thus $d_{v_4}= 2$.
 Let $T'=G-\{v_2,u_2\}$. Then $T'\in T(2(m-1),m-1)$ and $M-\{u_{2}v_2\}$ is a perfect matching of $T'$.

 By Lemma $\ref{lem4}$, $T'\in T(2(m-1),m-1)$, $D'\in \mathcal{O}(T')$ and $A(D')\bigcap A(D)=A(D')$, where $D\in \mathcal{O}(G)$. Then $M_1(D')\leq \frac{1}{2}[(m-1)^2+5(m-1)-4]$.
 Thus

$M_{1}(D)\leq M_{1}(D') + max\{\frac{1}{2}[d^{+}_{D}(v_2)+d^{-}_{D}(u_2)], \frac{1}{2}[d^{-}_{D}(v_2)+d^{+}_{D}(u_2)]\} + max\{\frac{1}{2}[d^{+}_{D}(v_2)+d^{-}_{D}(v_1)], \frac{1}{2}[d^{-}_{D}(v_2)+d^{+}_{D}(v_1)]\} + max\{\frac{1}{2}[d^{+}_{D}(v_2)+d^{-}_{D}(v_3)], \frac{1}{2}[d^{-}_{D}(v_2)+d^{+}_{D}(v_3)]\} + \frac{1}{2}max\{d^{+}_{D}(v_1)-1, d^{-}_{D}(v_1)-1\} + \frac{1}{2}max\{d^{+}_{D}(v_3)-1, d^{-}_{D}(v_3)-1\}
\leq M_{1}(D') + \frac{1}{2}[ d_{G}(v_2)+d_{G}(u_2)] + \frac{1}{2}[ d_{G}(v_2)+d_{G}(v_1)] + \frac{1}{2}[ d_{G}(v_2)+d_{G}(v_3)] +
\frac{1}{2}(d_{G}(v_1)-1) + \frac{1}{2}(d_{G}(v_3)-1)
\leq \frac{1}{2}[(m-1)^{2}+5(m-1)-4+8+10]
=\frac{1}{2}[m^{2}+3m+10].$


 Since $\frac{1}{2}[m^2+7m-2]-\frac{1}{2}[m^2+3m+10]=\frac{1}{2}[4m-12]\geq 0$, $M_{1}(D)\leq\frac{1}{2}[m^2+3m+10]\leq\frac{1}{2}[m^2+7m-2]$ with equality if and only if $D\in\{U^{(5)}_{6,3}, U^{(6)}_{6,3}\}$.
 Consequently, the result holds.
\end{proof}

\section{Proof of Theorem $\ref{the11}$}
In this section we give a proof of Theorem $\ref{the11}$. For this we need the following results:
\begin{lemma}\cite {AY04}\label{lem9}
 Let $G\in U(n,m)$ with $G \neq C_n$, where $n>2m$. Then there is a maximum matching $M$ of $G$ and a pendent vertex $u$ such that  is not  $M$-saturated.

\end{lemma}

\begin{lemma}\label{lem10}
 Let $n$ and $m$ be integers with $2\leq m\leq \lfloor\frac{n}{2}\rfloor$ and $n>2m$. Then $$\frac{1}{2}\left[n^{2}+(-2m+3)n+m^{2}+m-2\right]>2n$$

\end{lemma}
\begin{proof}
Let$$f(n,m)=\frac{1}{2}\left[n^{2}+(-2 m+3) n+m^{2}+m-2\right]-2n$$
then $$\frac{\partial f}{\partial m}=\frac{1}{2}(2m+1-2n)<0$$
When $n$ is even, $\lfloor\frac{n}{2}\rfloor=\frac{n}{2}$. Since $n>2m$ i.e., $m<\frac{n}{2}$,  $2\leq m<\frac{n}{2}$.
Hence $f(n,m)\geq f(n,\frac{n-2}{2})$.
Let $h(n)=f(n,\frac{n-2}{2})=\frac{1}{8}n^2+\frac{1}{4}n-1$. Since $h'(n)=\frac{n}{4}+\frac{1}{4}> 0$,  $h(n)\geq h(5)=\frac{27}{8}>0$.
Consequently, $f(n,m)\geq f(n,\frac{n}{2}-1)>0$, i.e. $\frac{1}{2}\left[n^{2}+(-2 m+3) n+m^{2}+m-2\right]>2n$.

When $n$ is odd, $\lfloor\frac{n}{2}\rfloor=\frac{n-1}{2}$.
 Since $f(n,\frac{n-1}{2})=\frac{1}{2}\left[n^{2}+(4-n) n+\frac{1}{4}(n-1)^{2}+\frac{n-1}{2}-2\right]-2n=-\frac{9}{8}+\frac{n^2}{8} $ and $n\geq 2m \geq4$,
  $f(n,\frac{n-1}{2})\geq 2-\frac{8}{9}>0$.
 Consequently, the results holds.
\end{proof}

We are now ready to give a proof of Theorem $\ref{the11}$.\\
\textbf{Proof of Theorem \ref{the11}}.
\begin{proof} 
We will prove by induction on $n$.

If $n = 2m$,  by Theorem 1, the result holds.

 If $n > 2m$. Suppose that the result holds for orientations of all unicyclic graphs on less than $n$ vertices.

  Let $G\in U(n, m)$. If $G=C_n$, then $D\in \mathcal{O}(C_n)$, by Lemma $\ref{lem2}$ and Lemma $\ref{lem10}$, $M_{1}(D)\leq \frac{M_{1}(C_n)}{2}=2n< \frac{1}{2}\left[n^{2}+(-2 m+3) n+m^{2}+m-2\right]$. The result holds.

 If $G\neq C_n$. By Lemma $\ref{lem9}$, $G$ has a maximum matching $M$ and a pendent vertex $u$ such that $u$ is not $M$-saturated. Then $G'=G-\{u\}\in U(n- 1,m)$. Let $D'\in\mathcal{O}(G')$ and $A(D')\cap A(D)=A(D')$.

By the induction hypothesis, it is obvious that $$M_{1}(D')\leq \frac{1}{2}\left[(n-1)^{2}+(-2 m+3) (n-1)+m^{2}+m-2\right].$$
\indent Let $v\in N_{G}(u)$ and $N_{G}(v)=\{u_1,u_2,\cdot\cdot\cdot,u_{s+1}\}$, where $s\geq 1$ and $u_1=u$. Since $M$ contains at most one of the edges $vu_i$ for $i = 2,3,\cdot\cdot\cdot,s+1$ and there are $n-m$ edges of $G$ outside $M$, it is obvious that $s\leq n-m$.\\
\indent If $uv\in A(D)$, then $\frac{1}{2}[d_{D}^{+}(u)+d_{D}^{-}(v)]\leq \frac{1}{2}[d_{G}(u)+d_{G}(v)]$. If $vu\in A(D)$, then $\frac{1}{2}[d_{D}^{-}(u)+d_{D}^{+}(v)]\leq \frac{1}{2}[d_{G}(u)+d_{G}(v)]$. Hence, $\max \{\frac{1}{2}[d_{D}^{+}(u)+d_{D}^{-}(v)],\frac{1}{2}[d_{D}^{-}(u)+d_{D}^{+}(v)]\}\leq \frac{1}{2}[d_{G}(u)+d_{G}(v)]$.\\
 \indent If $uv\in A(D)$, then $d_{D}^{-}(v)=d_{D'}^{-}(v)+1$, $d_{D}^{+}(v)=d_{D'}^{+}(v)$. Since $A(D')\cap A(D)=A(D')$,  without lost of generality suppose that $d_{D}^{+}(u_i)=d_{D'}^{+}(u_i)$, where $i=2,\cdot\cdot\cdot,d_{D}^{-}(v)$;  $d_{D}^{-}(u_j)=d_{D'}^{-}(u_j)$, where $j=d_{D}^{-}(v)+1,\cdot\cdot\cdot,d_{G}(v)$, we have    $d_{D}^{+}(u_i)+d_{D}^{-}(v)=d_{D'}^{+}(u_i)+d_{D'}^{-}(v)+1$, where $i=2,\cdot\cdot\cdot,d_{D}^{-}(v)$;  $d_{D}^{-}(u_j)+d_{D}^{+}(v)=d_{D'}^{-}(u_j)+d_{D'}^{+}(v)$, where   $j=d_{D}^{-}(v)+1,\cdot\cdot\cdot,d_{G}(v)$.

  Similarly to $vu\in A(D)$.
Thus


\begin{equation*}
\begin{aligned}
M_{1}(D) \leqslant& M_{1}(D')+\max \{\frac{1}{2}[d_{D}^{+}(u)+d_{D}^{-}(v)],\frac{1}{2}[d_{D}^{-}(u)+d_{D}^{+}(v)]\}+\max\{\frac{1}{2}\sum_{i=2}^{d_{D}^{+}(v)}[d_{D}^{-}(u_i)\\
&+d_{D}^{+}(v)-(d_{D'}^{-}(u_i)+d_{D'}^{+}(v))]+\frac{1}{2}\sum_{j=d_{D}^{+}(v)+1}^{d_{G}(v)}[d_{D}^{+}(u_j)+d_{D}^{-}(v)-(d_{D'}^{+}(u_j)+d_{D'}^{-}(v))],\\
&\frac{1}{2}\sum_{i=2}^{d_{D}^{-}(v)}[d_{D}^{+}(u_i)+d_{D}^{-}(v)-(d_{D'}^{+}(u_i)+d_{D'}^{-}(v))]
+\frac{1}{2}\sum_{j=d_{D}^{-}(v)+1}^{d_{G}(v)}[d_{D}^{-}(u_j)
+d_{D}^{+}(v)\\
&-(d_{D'}^{-}(u_j)+d_{D'}^{+}(v))]\} \\
\leqslant& M_{1}(D')+\frac{1}{2}[d_{G}(u)+d_{G}(v)]+\frac{1}{2}\max\{d_{D}^{+}(v)-1,d_{D}^{-}(v)-1\}\\ 
\leqslant& M_{1}(D')+\frac{1}{2}[d_{G}(u)+d_{G}(v)]+\frac{1}{2}(d_{G}(v)-1) \\
\leqslant& M_{1}(D')+s+1 \\
\leqslant& \frac{1}{2}\left[(n-1)^{2}+(-2m+3)(n-1)+m^{2}+m-2\right]+n-m+1\\
 =& \frac{1}{2}\left[n^{2}+(-2 m+3)n+m^{2}+m-2\right]
\end{aligned}
\end{equation*}
with equality if and only if $M_{1}(D')= \frac{1}{2}\left[(n-1)^{2}+(-2 m+3) (n-1)+m^{2}+m-2\right]$, $\max \{\frac{1}{2}[d_{D}^{+}(u)+d_{D}^{-}(v)],\frac{1}{2}[d_{D}^{-}(u)+d_{D}^{+}(v)]\}=\frac{1}{2}[d_{G}(u)+d_{G}(v)]$, $\frac{1}{2}\max \{[d_{D}^{+}(v)-1],[d_{D}^{-}(v)-1]\}=\frac{1}{2} [d_{G}(v)-1]$ and $s=n-m$,
or equivalently, $D'\in U_{n-1,m}^{*}$ and $\{d_{D}^{+}(u)=d_{G}(u),d_{D}^{-}(v)=d_{G}(v)\}$ or $\{d_{D}^{+}(v)=d_{G}(v),d_{D}^{-}(u)=d_{G}(u)\}$, i.e. $D\in U_{n,m}^{*}$. The result holds.
\end{proof}

\textbf{Acknowledgment}.
 This work is supported by the Hunan Provincial Natural Science Foundation of China (2020JJ4423), the Department of Education of Hunan Province (19A318) and the National Natural Science Foundation of China (11971164).


\end{document}